\documentclass[twoside, 11pt]{article}
\usepackage{amsmath,amssymb,amsthm,mathrsfs}
\usepackage[margin=2.5cm]{geometry} 
\usepackage{lipsum}
\usepackage{titlesec,hyperref}
\usepackage{fancyhdr}
\usepackage[numbers,sort&compress]{natbib}
\usepackage{color}

\fancypagestyle{plain}
{
\fancyhf{}

}

\titleformat{\subsection}{\it}{\thesubsection.\enspace}{1.5pt}{}
\titleformat{\subsubsection}{\it}{\thesubsubsection.\enspace}{1.5pt}{}

\newtheorem{theo}{Theorem}[section]
\newtheorem{lemm}[theo]{Lemma}
\newtheorem{defi}[theo]{Definition}

\numberwithin{equation}{section}

\allowdisplaybreaks

\def\th2{\frac{\theta}{2}}

\begin{document}

\title{On a nonisothermal  ideal gas Navier-Stokes-Fourier equations\hspace{-4mm}}
\author{ Boling Guo$^1$, Binqiang Xie$^2*$}
\date{}
\maketitle
\begin{center}
\begin{minipage}{120mm}
\emph{\small $^1$Institute of Applied Physics and Computational Mathematics, China Academy of Engineering Physics,
 Beijing, 100088, P. R. China \\
$^2$Graduate School of China Academy of Engineering Physics, Beijing, 100088, P. R. China  }
\end{minipage}
\end{center}

\footnotetext{Email: \it gbl@iapcm.ac.cn(B.L.Guo), \it xbq211@163.com(B.Q.Xie).}
\date{}

\maketitle

\begin{abstract}

In this paper we are concerned with a non-isothermal compressible Navier-Stokes-Fourier model with density dependent viscosity that vanish on the vacuum. We prove sequential stability of variational weak solutions in periodic domain $\Omega=T^{3}$. The main point is that the pressure is given by $P=R\rho \theta$.

\vspace*{5pt}

\vspace*{5pt}
\end{abstract}


\section{Introduction}
\quad\quad A compressible and heat-conducting fluid governed by the Navier-Stokes-Fourier equations satisfies the following system in $R_{+}\times \Omega$:
\begin{gather}
  \partial_{t}\rho+{\rm div}(\rho u) =0,\label{1.1} \\
  \partial_{t}(\rho u)+{\rm div}(\rho u\otimes u)
  +\nabla P={\rm div}\mathbb{S}, \label{1.2}\\
  \partial_{t} (\rho E) + {\rm div}(\rho E u)+{\rm div} q+ {\rm div}(P u)={\rm div} (\mathbb{S}u),\label{1.3}
\end{gather}
where the functions $\rho, u,\theta $ represent the density,the velocity field, the absolute temperature. $P$ stands for the pressure, $\mathbb{S}$ denotes the viscous stress tensor. $\rho E=\rho e+ \frac{\rho |u|^{2}}{2}$ the total energy, $e$ the internal energy. $q$ the heat flux. Eqs. \eqref{1.1}, \eqref{1.2}, \eqref{1.3} respectively express the conservation of mass, momentum and total energy.

Our analysis is based on the following physically grounded assumptions:
\begin{itemize}
\item The viscosity stress tensor $\mathbb{S}$ is determined by the Newton's rheological law
\begin{equation}\label{1.4}
\mathbb{S}=2\mu(\rho) D(u)+\lambda(\rho) {\rm div}_{x} u \mathbb{I},
\end{equation}
where $3\lambda + 2\mu\geq 0$ and $D(u)= \frac{1}{2}(\nabla u+ \nabla^{T} u)$ denotes the strain rate tensor, we require $\lambda(\rho)= 2(\rho \mu^{\prime}(\rho)- \mu(\rho))$. For simplicity, we only consider a particular case $\mu(\rho)=\rho, \lambda(\rho)=0$.

\item A key element of the system \eqref{1.1}-\eqref{1.3} is pressure $P$, which obeys the following equation of state:
\begin{equation}\label{1.5}
P(\rho,\theta)=R\rho\theta,
\end{equation}
where R is the perfect gas constant, for simplicity, we set $R=1$. This assumption means ideal gas given by Boyle's law.
\end{itemize}

\begin{itemize}
\item In accordance with the second thermodynamics law, the form of the internal energy reads:
\begin{equation}\label{1.6}
e=C_{\nu}\theta ,
\end{equation}
where $C_{\nu}$ is termed the specific heat at constant volume, for simplicity, we set $C_{\nu}=1$.

\item The heat flux $q$ is expressed through the classical Fourier's law:
\begin{equation}\label{1.7}
q=-\kappa\nabla \theta ,
\end{equation}
where the heat conducting coefficient $\kappa$ is assumed to satisfy:
\begin{equation}\label{1.8}
\kappa(\rho,\theta)= \kappa_{0}(\rho,\theta)(1+\rho)(1+\theta^{a}),
\end{equation}
where $a\geq 2$, $\kappa_{0}$ is a continuous function of temperature and density satisfying:
$ C_{1} \leq \kappa_{0} (\rho, \theta) \leq \frac{1}{C_{1}}$, for some positive $C_{1}$.
\end{itemize}

To complete the system \eqref{1.1}-\eqref{1.3}, the initial conditions are given by
\begin{equation}\label{1.9}
\rho(0,\cdot)=\rho_{0}, (\rho u)(0,\cdot)=m_{0}, \theta(0,\cdot)=\theta_{0},
\end{equation}
together with the compatibility condition:
\begin{equation}\label{1.10}
m_{0}=0~~on~the~set~~\{x\in \Omega|\rho_{0}(x)=0\}.
\end{equation}

If the solutions are smooth, the temperature equation is easier to deduce as follows:
\begin{equation}\label{1.11}
\partial_{t} (\rho \theta)+ {\rm div} (\rho \theta u) =\mathbb{S}: \nabla u+ {\rm div} (\kappa\nabla \theta)-\rho \theta {\rm div} u .
\end{equation}
Finally, assuming the pressure and internal energy, we can define the specific entropy through Gibbs relationship.
Accordingly, the temperature equation may be put into an equivalent form of the entropy equation:
\begin{equation}\label{1.12}
\partial_{t} (\rho s)+ {\rm div} (\rho s u)-{\rm div}\bigg(\frac{\kappa \nabla \theta}{\theta}\bigg) =\frac{\mathbb{S}: \nabla u}{\theta}+ \frac{\kappa|\nabla \theta|^{2}}{\theta^{2}}.
\end{equation}
where $s= \ln \theta-\ln \rho$.

In the follows we give the definition of a variational solution to \eqref{1.1}-\eqref{1.10}.

\begin{defi}
We call $(\rho, u, \theta)$ is as  a varational  weak solution to the problem \eqref{1.1}-\eqref{1.10}, if the following is satisfied.

(1)the density $\rho$ is a non-negative function satisfying the internal identity
\begin{equation}\label{1.13}
\int_{0}^{T} \int_{\Omega} \rho \partial_{t} \phi + \rho u\cdot \nabla \phi dx dt+ \int_{\Omega} \rho_{0} \phi(0)dx=0,
\end{equation}
for any test function $\phi\in \mathcal{D}([0,T)\times \overline{\Omega})$.

(2) The momentum equation holds in $D^{\prime}((0,T)\times\Omega)$, that means,
\begin{equation}\label{1.14}
\begin{aligned}
&\int_{\Omega}m_{0} \phi(0) dx+\int_{0}^{T} \int_{\Omega} \rho u \cdot \partial_{t} \phi + \rho (u\otimes u): \nabla \phi+ P {\rm div} \phi dx dt\\
&= \int_{0}^{T} \int_{\Omega} \mathbb{S}: \nabla \phi dx dt,~for~any ~\phi\in \mathcal{D}([0,T)\times \overline{\Omega}),
\end{aligned}
\end{equation}

(3) If we admit there are variational solutions that may fail to satisfy the weak formulation of the total energy balance, thus this can be expressed by the variational principle of entropy production:
\begin{equation}\label{1.15}
\int_{\Omega} \rho_{0} s_{0} \phi(0)dx +\int_{0}^{T} \int_{\Omega} \rho s \partial_{t} \phi + \rho s u \cdot \nabla \phi- \frac{\kappa \nabla \theta}{\theta} \nabla \phi dx dt+ <\sigma,\phi>=0
\end{equation}
It is satisfied for any smooth function $\phi(x,t)$, such that $\phi\geq 0$ and $\phi(T,\cdot)=0$, where $\sigma\in \mathcal{M}^{+}((0,T)\times \Omega)$ is a nonnegative measure such that
\begin{equation*}
\sigma\geq \frac{\mathbb{S}: \nabla u}{\theta}+ \frac{\kappa|\nabla \theta|^{2}}{\theta^{2}}
\end{equation*}

(4)The global balance of total energy
\begin{equation}\label{1.16}
\int_{\Omega} (\rho E)^{0} \phi(0) dx  +\int_{0}^{T}\int_{\Omega} \rho E \partial_{t} \phi dxdt=0
\end{equation}
holds for any smooth function $\phi(t)$, such that $\phi(T)=0$.
\end{defi}

Now, we are ready to formulate the main result of this paper.

\begin{theo}[Stability]\label{th1.1}
Let $\Omega$ be the periodic box $T^{3}$.  Assume that
the pressure $P$, the conductivity coefficient $\kappa$ and the viscosity coefficient $\mu$  satisfy the condition \eqref{1.4}-\eqref{1.8}. Let $(\rho_{n}, u_{n},\theta_{n})_{n\in \mathbb{N}}$ be sequence of weak solutions of \eqref{1.1}-\eqref{1.3} satisfying entropy inequalities \eqref{2.1}, \eqref{2.3} and \eqref{2.14}, with  initial data
\begin{equation*}
\rho_{n}|_{t=0}=\rho_{0}^{n}(x), ~~\rho_{n} u_{n}|_{t=0}=m_{0}^{n}(x)= \rho_{0}^{n}(x) u_{0}^{n}(x),~~\theta_{n}|_{t=0}= \theta_{0}^{n}(x),
\end{equation*}
where $\theta_{0}^{n}$, $u_{0}^{n}$ and $\theta_{0}^{n}$ are such that
\begin{equation}\label{1.17}
\rho_{0}^{n}\geq 0,~~\rho_{0}^{n} \rightarrow \rho_{0}~~~\in~ L^{1}(\Omega),~~\rho_{0}^{n}u_{0}^{n}\rightarrow \rho_{0}u_{0}~~~\in~ L^{1}(\Omega),~~\theta_{0}^{n} \rightarrow \theta_{0}~~~\in~ L^{1}(\Omega),
\end{equation}
and satisfying the following bounds (with $C$ constant independent on $n$):
\begin{equation}\label{1.18}
\int_{\Omega}\frac{|m_{0}|^{2}}{\rho_{0}}+ \rho_{0} \theta_{0}dx <C,~~\int_{\Omega}\nabla \sqrt{\rho_{0}} dx<C, ~~~ \int_{\Omega}\rho_{0} s_{0}<C ,
\end{equation}
Then, up to a subsequence, $(\rho_{n},\sqrt{\rho_{n}}u_{n},\theta_{n})$ converges strongly to a weak solution of \eqref{1.1}-\eqref{1.3} satisfying entropy inequalities \eqref{2.1}, \eqref{2.3} and \eqref{2.15}.
\end{theo}

For this full compressible Navier-Stokes system in the constant viscosity case, Lions \cite{Lions} and Feireisl\cite{Firesel},\cite{Firesel2} proved the global existence of variational weak solutions. Such an existence result is obtained for specific pressure laws, given by general pressure equation
\begin{equation*}
P(\rho,\theta)= P_{b}(\rho)+ \theta P_{\theta}(\rho).
\end{equation*}
 Unfortunately, the perfect gas equation of state is not covered by this result. Namely the dominant role of the first, barotropic pressure $P_{b}$ is one of the key argument to obtain such an existence result.

As for the density depending viscosities case, the existence of global weal solutions to the Navier-Stokes equations for viscous compressible and heat conducting fluids was proved by D. Bresch and B. Desjardins \cite{Desjardin}. The equation of state is ideal polytropic gas type:
 \begin{equation*}
 P =R\rho \theta+ P_{c}(\rho),
 \end{equation*}
However, they still need additional cold pressure assumption $P_{c}$.  Therefore, Our aim in this work is to remove additional assumption on the equation of state. For related B-D entropy inequality, we refer the paper \cite{Bresch},\cite{Desjardins},\cite{Mellet}.

In order to prove the stability of variational weak solutions, the first step is to obtain suitable a priori bounds on $(\rho_{n},u_{n},\theta_{n})$.  The next step is to obtain compactness on  $(\rho_{n},u_{n},\theta_{n})$ in suitably strong topologies and prove that the limit $(\rho,u,\theta)$ satisfies Eqs.\eqref{1.1}-\eqref{1.10} in the variational sense.

In the forthcoming we will use the idea of Li and Xin \cite{Li} to construct approximate solution, thus we can complete the global existence of variational weak solution for non-isothermal Navier-Stokes-Fouier system.

This paper is organized as follows. In section $2$, we deduce a priori estimates from \eqref{1.1}. In section $3$, we establish the compactness of solutions $(\rho_{n},u_{n},\theta_{n})$. In section $4$, we will prove the proof of theorem 1.2 using Aubin-Lions Lemma.

\section{A priori bounds}
\quad\quad In this section, we collect the available a priori estimates for  sequence of smooth functions $\{\rho_{n}, u_{n}, \theta_{n} \}$ solving \eqref{1.1}-\eqref{1.3}. As mentioned above, assuming smoothness of solutions, we will deduce enough estimates to establish the compactness of solutions. The following estimates are valid for each $n=1,2,...$ but we skip the subindex it does not lead to any confusion.

\subsection{\bf{Estimates based on the maximum principle}}
First, it is easier to know that the total mass of the fluid is as constant of motion, i.e.
\begin{equation}\label{2.1}
\int_{\Omega} \rho(x,t) dx= \int_{\Omega} \rho_{0}dx=M_{0},~~~for~~t\in[0,T].
\end{equation}
Moreover maximum principle can be applied to the continuity equation in order to show that $\rho_{n}>c(n)\geq 0$, more precisely
\begin{equation}\label{2.2}
\rho_{n}(x,t)\geq \inf_{x\in \Omega} \rho_{n}^{0} exp(-\int_{0}^{T}\|{\rm div} u_{n}\|_{L^{\infty}(\Omega)}dt),
\end{equation}
in particular $\rho>0$.

Next, by a similar reasoning we can prove non-negatively of $\theta$ on $[0,T]\times \Omega$.
\begin{lemm}
Assume that $\theta=\theta_{n}$ is a smooth solution of (1.1), then
\begin{equation}\label{2.3}
\theta(t,x)>c(N)\geq 0,~~~for ~(t,x)\in [0,T]\times \Omega.
\end{equation}
\end{lemm}

\subsection{\bf{The energy-entropy estimates}}
The physical energy inequality (involving the internal and kinetic energy) is classical in the full compressible Navier-Stokes equations which is shown in the following:
\begin{lemm}[Physical energy estimates]\label{lm4.4}
\begin{equation} \label{2.4}
\int_{\Omega} \rho(\frac{|u|^{2}}{2}+ \theta)(t,x) dx \leq \int_{\Omega}(\frac{|m_{0}|^{2}}{\rho_{0}}+ \rho_{0} \theta_{0})dx
\end{equation}
\end{lemm}
\begin{proof}
Integrate \eqref{1.3} with respect to the space variable and employ the periodic boundary conditions.
\end{proof}
Assume the initial total energy is finite, we immediately obtain the following bounds:
\begin{equation}\label{2.5}
\|\sqrt{\rho} u\|_{L^{\infty}(0,T;L^{2}(\Omega))} \leq C,  \| \rho \theta\|_{L^{\infty}(0,T;L^{1}(\Omega))}\leq C,
\end{equation}

Next, we give some temperature estimates from the entropy equation:
\begin{lemm}[Entropy estimates]\label{lm4.4}
Assume that $\rho_{0} s_{0} \in L^{1}(\Omega)$, Then for all $T \geq 0$, one has
\begin{equation} \label{2.6}
\int_{0}^{T}\int_{\Omega} \frac{\kappa|\nabla \theta|^{2}}{\theta^{2}}+ \frac{2\rho |D(u)|^{2}}{\theta} dx dt\leq \int_{\Omega}\rho s + |\rho_{0} s_{0}|dx
\end{equation}
\end{lemm}
\begin{proof}
Integrate \eqref{1.12} with respect to the space variable and employ the periodic boundary conditions.
\end{proof}

The first term on the right hand of \eqref{2.6} can be estimated by:
\begin{equation}\label{2.7}
\int_{\Omega} \rho s dx \leq \int_{\Omega}\rho \ln\theta dx- \int_{\Omega} \rho \ln \rho dx,
\end{equation}
Multiplying the mass equation by $1+\ln \rho$, we get
\begin{equation}\label{2.8}
\partial_{t}(\rho \ln \rho) + {\rm div}(\rho \ln\rho u)+ \rho {\rm div} u=0,
\end{equation}
Therefore, the right hand of \eqref{2.6} can be estimated by
\begin{equation}\label{2.9}
\begin{aligned}
\int_{\Omega} \rho s dx &\leq \int_{\Omega} \rho \theta dx + \int_{\Omega} |\rho_{0} \ln \rho_{0}| dx + \int_{0}^{T} \int_{\Omega} \rho {\rm div} u dx dt \\
& \leq C+ \int_{0}^{T} \int_{\Omega} \sqrt{\frac{\rho}{\theta}}|{\rm div} u| \sqrt{\rho \theta} dx dt\\
& \leq C + \varepsilon \int_{0}^{T} \int_{\Omega}\frac{\rho|D(u)|^{2}}{\theta} dx dt + C(\varepsilon)
\int_{0}^{T} \int_{\Omega}\rho \theta dx dt,
\end{aligned}
\end{equation}

Hence if $\rho_{0} s_{0}$ and $\rho_{0} \ln \rho_{0}$ belong to $L^{1}(\Omega)$, then the component of following quantities $\sqrt{\rho} D(u) / \sqrt{\rho}$, $(\sqrt{\rho}+1) \nabla \theta^{a/2}$, $(\sqrt{\rho}+1) \nabla \ln \theta$ are bounded in $L^{2}(\Omega\times (0,T))$. We note that the last two bounds involving the temperature gradient provide the following useful estimates:
\begin{equation}\label{2.10}
(\sqrt{\rho}+1) \nabla \theta^{\alpha} \in L^{2}(\Omega\times (0,T)),~for~all~\alpha~such~ that ~0\leq\alpha \leq a/2.
\end{equation}

Now, we will derive some estimates on  velocity and associated effective B-D entropy energy.
\begin{lemm}[The kinetic energy estimates]\label{lm4.4}
\begin{equation} \label{2.11}
\frac{d}{dt}\int_{\Omega} \frac{1}{2}\rho |u|^{2}dx+ \int_{\Omega} 2\rho |D(u)|^{2} dx=\int_{\Omega} P {\rm div} u dx,
\end{equation}
\end{lemm}
\begin{proof}
Multiply the momentum equation \eqref{1.2} by $u$ and integrate over $\Omega$.
\end{proof}

\begin{lemm}[B-D effective energy estimates]\label{lm4.4}
\begin{equation} \label{2.12}
\begin{aligned}
&\frac{d}{dt}\int_{\Omega} \frac{1}{2}\rho |u+ 2\nabla \ln \rho|^{2} dx + \int_{\Omega} 2\rho |A(u)|^{2}dx+ 2\int_{\Omega} \frac{|\nabla \rho|^{2} \theta}{\rho} dx\\
& = \int_{\Omega} P {\rm div} u dx- 2 \int_{\Omega} \nabla \rho\cdot \nabla \theta dx,
\end{aligned}
\end{equation}
\end{lemm}
\begin{proof}
The idea of the proof is from the original work of Bresch and Desjaedins. For more detail we refer to \cite{Desjardin}.
\end{proof}

In order to get enough a priori estimates from Lemma 2.4 and 2.5, we have to control the right-hand side terms of \eqref{2.11} and \eqref{2.12}.
\begin{lemm}[$\int_{\Omega} P {\rm div} u dx$]\label{lm4.4}
\begin{equation} \label{2.13}
\begin{aligned}
&\int_{\Omega} P {\rm div} u dx \\
\leq & \varepsilon \int_{\Omega} \rho |{\rm div} u|^{2}dx + C(\varepsilon) \int_{\Omega} \int_{\Omega} \rho \theta^{2} dx\\
\leq &\varepsilon \int_{\Omega} \rho |{\rm div} u|^{2}dx + C(\varepsilon) \|\theta\|_{L^{3}}^{2} \| \nabla \sqrt{\rho}\|_{L^{2}}^{2},
\end{aligned}
\end{equation}
\end{lemm}

\begin{lemm}[$\int_{\Omega} \nabla \rho \cdot \nabla \theta dx$]\label{lm4.4}
\begin{equation} \label{2.14}
\begin{aligned}
&\int_{\Omega} \nabla \rho \cdot \nabla \theta dx \\
\leq & C\int_{\Omega} \frac{\rho \theta^{2}}{\kappa} |\nabla \sqrt {\rho}|^{2}dx + C \int_{\Omega} \frac{\kappa |\nabla \theta|^{2}}{\theta^{2}}\\
\leq &  C \int_{\Omega}  |\nabla \sqrt {\rho}|^{2}dx + C,
\end{aligned}
\end{equation}
\end{lemm}

Thus,  by taking $\varepsilon$ small enough, \eqref{2.10} and Sobolev inequality, $\theta\in L^{2}([0,T]; L^{6}(\Omega))\cap L^{2}([0,T]; L^{6}(\Omega))$, it is possible to get some a priori estimates via Gronwall's inequality.

From above we get the space compactness of the density and temperature, therefore the strong convergence of the density and temperature can be derived. But there  still is lack of information about the velocity. To this goal, in the following we can prove the so-called Mellet-Vasseur type estimate.

\begin{lemm}[Mellet-Vasseur type estimate]\label{lm4.4}
\begin{equation} \label{2.15}
\sup_{0\leq t\leq T} \int_{\Omega} \rho (1+|u|^{2}) \ln (1+|u|^{2}) dx \leq C,
\end{equation}
\end{lemm}
\begin{proof}
Multiplying momentum equation \eqref{1.2} by $(1+\ln(1+|u|^{2}))u$ and integrating over $\Omega$ lead to
\begin{equation*}
\begin{aligned}
&\frac{1}{2} \frac{d}{dt} \int_{\Omega} \rho (1+|u|^{2}) \ln (1+|u|^{2}) dx + \int_{\Omega}  (1+\ln (1+|u|^{2}) \rho |D(u)|^{2}dx\\
& \leq C\int_{\Omega} \rho |D(u)|^{2} dx - \int_{\Omega} (1+\ln (1+|u|^{2})u\cdot \nabla (\rho \theta) dx,
\end{aligned}
\end{equation*}
where the last term on the right side can be estimated as follows:
\begin{equation*}
\begin{aligned}
&| \int_{\Omega} (1+\ln (1+|u|^{2})u\cdot \nabla (\rho \theta) dx|\\
\leq & \int_{\Omega} (1+\ln (1+|u|^{2}) {\rm div}u \rho \theta dx + \int_{\Omega} \frac{2u_{i}u_{k}}{1+|u|^{2}} \partial_{i} u_{k} \rho \theta dx \\
\leq &\varepsilon \int_{\Omega} (1+\ln (1+|u|^{2}) \rho |Du|^{2} dx+ C \int_{\Omega} (1+\ln (1+|u|^{2}) \rho \theta^{2} dx\\
 +& C\| \sqrt{\rho} \nabla u \|_{L^{2}(\Omega)}
\|\theta\|_{L^{3}(\Omega)} \|\sqrt{\rho} \|_{L^{6}(\Omega)}, \\
\leq &C + C \int_{\Omega} (1+|u|) \rho \theta^{2}dx\\
\leq &C + C \int_{\Omega}\rho |u|^{2} dx + C\int_{\Omega} \rho \theta^{4}dx\\
\leq & C
\end{aligned}
\end{equation*}
\end{proof}

\section{Compactness of $\rho_{n}, \sqrt{\rho_{n}}u_{n}, \theta_{n}$}
\quad\quad  We recall that the initial data must satisfy \eqref{2.1}, \eqref{2.3} and \eqref{2.15} to make use of all the inequalities presented in the previous section. More precisely, we take
\begin{equation}\label{3.1}
\rho_{0}^{n}~~ is~~ bounded~~ in~~ L^{1}(\Omega), \rho_{0}^{n}\geq 0~~ a.e.~~ in~~ \Omega ,
\end{equation}
\begin{equation}\label{3.2}
\rho_{0}^{n}|u_{0}^{n}|^{2}= |m_{0}^{n}|^{2}/\rho_{0}^{n} ~~is~~ bounded~~ in~~ L^{1}(\Omega),
\end{equation}
\begin{equation}\label{3.3}
\rho_{0}s_{0}~~and~~\rho_{0} \ln \rho_{0}~~is~~ bounded~~ in~~ L^{1}(\Omega),
\end{equation}
\begin{equation}\label{3.4}
\nabla \sqrt{\rho_{0}^{n}} ~~is~~ bounded~~ in~~ L^{2}(\Omega),
\end{equation}
\begin{equation}\label{3.5}
\int_{\Omega} \rho_{0}^{n} \frac{|u_{0}^{n}|^{2}}{2}\ln(1+ |u_{0}^{n}|^{2})dx< C ,
\end{equation}
Using inequalities \eqref{2.1}, \eqref{2.3},\eqref{2.11},\eqref{2.12} and \eqref{2.15}, we deduce the following estimates, which we shall use throughout the proof of Theorem 1.2:
\begin{equation}\label{3.6}
\|\sqrt{\rho_{n}} u_{n}\|_{L^{\infty}(0,T;L^{2}(\Omega))} \leq C,
\end{equation}
\begin{equation}\label{3.333}
\|\sqrt{\rho_{n}} \nabla u_{n}\|_{L^{2}(0,T;L^{2}(\Omega))} \leq C,
\end{equation}
\begin{equation}\label{3.7}
\|\rho_{n}\|_{L^{\infty}(0,T;L^{1}(\Omega))} \leq C,
\end{equation}
\begin{equation}\label{3.8}
\|\nabla \sqrt{\rho_{n}}\|_{L^{\infty}(0,T;L^{2}(\Omega))} \leq C,
\end{equation}
\begin{equation}\label{3.9}
\|(1+\sqrt{\rho_{n}}) \nabla \theta_{n}^{\alpha}\|_{L^{\infty}(0,T;L^{2}(\Omega))} \leq C,
\end{equation}
\begin{equation}\label{3.10}
\int_{\Omega} \rho_{n} \frac{|u_{n}|^{2}}{2}\ln(1+ |u_{n}|^{2})dx< C ,
\end{equation}
Given above a priori bounds, we now intend to study the compactness of sequences of approximate solutions
 $\rho_{n}$, $\sqrt{\rho_{n}}u_{n}$ and $\theta_{n}$ and pass the limit in the nonlinear terms.

\subsection{\bf {Strong convergence of $\sqrt{\rho_{n}}$}}

\begin{lemm}
Up to a subsequence,
\begin{equation}\label{3.11}
\sqrt{\rho_{n}} \rightarrow \sqrt{\rho} ~a.e.~~ and~~~L^{2}_{loc}((0,T)\times \Omega)~~strong.
\end{equation}
In particular,
\begin{equation}\label{3.12}
\rho_{n}\rightarrow \rho ~~~~in~ C^{0}(0,T;L^{3/2}_{loc}(\Omega)),
\end{equation}
\end{lemm}
\begin{proof}
The proof is refer to Lemma 4.1 in \cite{Mellet}.
\end{proof}

\subsection{\bf {Strong convergence of $\sqrt{\rho_{n}}u_{n}$}}

\begin{lemm}
Up to a subsequence,
\begin{equation}\label{3.13}
\rho_{n} u_{n}\rightarrow \rho u ~~strongly~in~~ L^{2}(0,T; L^{p}_{loc}(\Omega)),~~for~p\in~[1,3/2).
\end{equation}
\begin{equation}\label{3.14}
\sqrt{\rho_{n}} u_{n}\rightarrow \sqrt{\rho} u ~~strongly~in~~ L^{2}_{loc}((0,T)\times\Omega),
\end{equation}
\end{lemm}
\begin{proof}
The proof is refer to Lemma 4.6 in \cite{Mellet}.
\end{proof}

\subsection{\bf {Strong convergence of the temperature }}
By \eqref{3.9} and  Sobolev imbedding gives the estimate of the norm of $\theta$ in $L^{2}(0,T;L^{6}(\Omega))$, and so, due to the boundedness of $\nabla \theta^{\frac{\alpha}{2}}$ in $L^{2}((0,T)\times \Omega)$, one gets
\begin{equation*}
\theta^{\frac{\alpha}{2}}\in L^{2}(0,T;W^{1,2}(\Omega)),
\end{equation*}

Therefore we deduce existence of a subsequence such that
\begin{equation}\label{3.15}
\theta_{n}\rightarrow\theta \ weakly \ in \ L^{2}(0,T;W^{1,2}(\Omega)),
\end{equation}
however, time-compactness cannot proved directly from the internal energy equation \eqref{1.11}. The reason for this is lack of control over a part of the heat flux proportional to $\rho_{n}\theta_{n}^{a}\nabla\theta_{n}$. This obstacle can be overcome by deducing analogous information from the entropy equation \eqref{1.12}.

We will first show that all of the terms appearing in the entropy balance \eqref{1.12} are nonnegative or belong to $W^{-1,p}((0,T)\times \Omega)$, for some $p>1$.
Indeed, first we recall that due to (2.8)
\begin{equation*}
 |\rho_{N}s_{n}|\leq C (\rho_{n} |\ln \theta_{n}|+ \rho_{n}|\ln \rho_{n}|)
\end{equation*}
and
\begin{equation*}
 |\rho_{n}s_{n}u_{n}|\leq C ( |\rho_{n}\ln \theta_{n}u_{n}|+ |\rho_{n}\ln \rho_{n}u_{n}|)
\end{equation*}
whence due to \eqref{3.6},\eqref{3.7} and \eqref{3.15} we deduce that
\begin{equation}\label{3.16}
 \{\rho_{n}s_{n}\}_{n=1}^{\infty} \ in \ bounded \ in \ L^{2}((0,T)\times \Omega),
\end{equation}
moreover
\begin{equation}\label{3.17}
 \{\rho_{n}s_{n}u_{n}\}_{n=1}^{\infty} \ in \ bounded \ in \ L^{2}(0,T;L^{\frac{6}{5}} (\Omega)),
\end{equation}
The entropy flux can be estimated as follows
\begin{equation}\label{3.18}
 |\frac{\kappa(\rho_{n},\theta_{n})\nabla\rho_{n}}{\theta_{n}}|\leq |\nabla\ln\theta_{n}|+|\rho_{n}\nabla\ln\theta_{n}|
 +|\theta_{n}^{\alpha-1}\nabla\theta_{n}|+|\rho_{n}\theta_{n}^{\alpha-1}\nabla\theta_{n}|
\end{equation}
where the most restrictive can be controlled as follows $|\rho_{n}\theta_{n}^{\alpha-1}\nabla\theta_{n}|\leq
|\sqrt{\rho_{n}}\theta_{n}^{\frac{\alpha}{2}}|
|\sqrt{\rho_{n}}\nabla\theta_{n}^{\frac{\alpha}{2}}|$, which is bounded on account of \eqref{3.9} provided $\rho_{n}\theta_{n}^{\alpha}$ is bounded in $L^{p}((0,T)\times \Omega)$ for $p>1$, uniformly with respect to n. Note that for $0\leq\beta\leq1$ we have $\rho_{n}\theta_{n}^{\alpha}
=(\rho_{n}\theta_{n})^{\beta}\rho_{n}^{1-\beta}\theta_{n}^{\alpha-\beta}$, where $(\rho_{n}\theta_{n})^{\beta},\rho_{n}^{1-\beta},\theta_{n}^{\alpha-\beta}$ are uniformly bounded in $L^{\infty}(0,T;L^{\frac{1}{\beta}}(\Omega)),L^{\infty}(0,T;L^{\frac{3}{1-\beta}}(\Omega)),
L^{\frac{\alpha}{\alpha-\beta}}(0,T;L^{\frac{3\alpha}{\alpha-\beta}}(\Omega))$, respectively. Therefore
\begin{equation}\label{3.19}
 \{\frac{\kappa(\rho_{n},\theta_{n})\nabla\rho_{n}}{\theta_{n}}\}_{n=1}^{\infty} \ is \ bounded \ in \ L^{p}(0,T;L^{q}(\Omega)).
\end{equation}
for p and q satisfying $\frac{1}{p}=\frac{\alpha}{\alpha-\beta},
\frac{1}{q}=\beta+\frac{1-\beta}{3}+\frac{\alpha-\beta}{3\alpha}$. In particular $p,q>1$ provided $0<\beta<\frac{2\alpha}{2\alpha-1}$.

We are now ready to proceed with the proof of the strong convergence of the temperature. To this end we will need the following variant of the Aubin-Lions Lemma.

\begin{lemm}\label{lem:0206}
Let $g^{n}$ converges weakly to g in $L^{p_{1}}(0,T;L^{p_{2}}(\Omega))$ and let $h^{n}$ converges weakly to h in $L^{q_{1}}(0,T;L^{q_{2}}(\Omega))$, where $1\leq p_{1},p_{2}\leq \infty$ and
\begin{equation}\label{3.20}
\frac{1}{p_{1}}+\frac{1}{q_{1}}=\frac{1}{p_{2}}+\frac{1}{q_{2}}=1
\end{equation}
Let us assume in addition that
\begin{equation}\label{3.21}
\frac{\partial g^{n}}{\partial t} \ is \ bounded \ in \ L^{1}(0,T;W^{-m,1}(\Omega)) \ for \
some \ m\geq0 \ independent\ of \ n.
\end{equation}
\begin{equation}\label{3.22}
\|h^{n}-h^{n}(\cdot+\xi,t)\|_{L^{q_{1}}(0,T;L^{q_{2}}(\Omega))}\rightarrow 0,\ uniformly \ in \ n.
\end{equation}
Then $g^{n}h^{n}$ converges to gh in the sense of distributions on $\Omega\times (0,T)$.
\end{lemm}

For the proof see \cite{Firesel}, Lemm5.1.

Taking $g^{n}=\rho_{n}s_{n}$ and $h^{n}=\theta_{n}$ we verify, due to \eqref{3.4}, that conditions \eqref{3.20} and \eqref{3.22} are satisfied with $p_{1},p_{2},q_{1},q_{2}=2$. Moreover for m sufficiently large $L^{1}(\Omega)$ is imbedded into $W^{-m,1}$, thus by the previous considerations, condition \eqref{3.21} is also fulfilled. Therefore, passing to the subsequences we may deduce that
\begin{equation*}
 \lim_{n\rightarrow\infty} \rho_{n}s(\rho_{n},\theta_{n})\theta_{n}=
 \overline{\rho s(\rho,\theta)}\theta.
\end{equation*}
On the other hand, $\rho_{n}$ converges to $\rho$ a.e. on $(0,T)\times \Omega$, hence
$\overline{\rho s(\rho,\theta)}\theta=\rho\overline{ s(\rho,\theta)}\theta$, in particular, we have that
\begin{equation}\label{3.23}
 C_{\mu}\overline{\rho\ln \theta}\theta-R\overline{\rho\ln \rho}=C_{\mu}\rho\overline{\ln \theta}\theta-R\rho\overline{\ln \rho}\theta
\end{equation}
Combining  weak convergence of the temperature with strong convergence of the density we identify
\begin{equation}\label{3.24}
R\rho\overline{\ln \rho}\theta=R\rho\ln \rho\theta
\end{equation}
so \eqref{3.24} implies that $\rho\overline{\ln \theta\theta}=\rho\overline{\ln \theta}\theta$. This in return yields that $\overline{\ln \theta\theta}=\ln \theta\theta$ a.e. on $(0,T)\times \Omega$,since $\rho>0$ a.e. on $(0,T)\times \Omega$, which, due to convexity of function $x \ln x$, gives rise to
\begin{equation}\label{3.25}
\theta_{n}\rightarrow\theta \ a.e.\ on \ (0,T)\times \Omega
\end{equation}

\section{Proof of the Theorem 1.2}
\quad\quad To finish the proof of Theorem 2.1, we need to check that the limit $\rho,u,\theta$ are indeed the weak solutions, as defined in the introduction. We will complete this proof by several steps.

Step 1. Convergence of the mass conservation equation. For the mass conservation, by the strong convergences of $\rho_{n}$ to $\rho$ in $C([0,T];L^{3/2}(\Omega))$ and the strong convergence of $\sqrt{\rho}u$ in $L^{2}(0,T; L^{2}_{loc}(\Omega))$,  the mass conservation equation \eqref{1.1} is satisfied in the sense of distribution.

Step 2. Convergence of the momentum conservation equation. For the momentum equation, the strong convergence of $\rho_{n}u_{n}$ and $\rho_{n}u_{n}\otimes u_{n}$ $L^{1}((0,T)\times \Omega)$ can ensure the passing to limit in the sense of distribution for the two corresponding term in the momentum conservation equation \eqref{1.2}. On the other hand, since $\rho_{n}$ and $\theta_{n}$ respectively converge strongly in $C(0,T;L^{3/2}(\Omega))$ and $L^{2}((0,T)\times \Omega)$, the term $\nabla(\rho_{n}\theta_{n})$ converges to the limit $\nabla(\rho\theta)$ in the sense of distribution. As for the convergence of the  viscous term, we rewrite this term as follows:
\begin{equation}\label{4.1}
\int_{0}^{T}\int_{\Omega} \rho_{n} \nabla u_{n} \phi dx dt=- \int_{0}^{T}\int_{\Omega} \sqrt{\rho_{n}} \sqrt{\rho_{n}} u_{n} \nabla \phi dx dt-  \int_{0}^{T}\int_{\Omega} \sqrt{\rho_{n}}u_{n} \nabla \sqrt{\rho_{n}}\phi dx dt,
\end{equation}
where $\phi$ be a test function. Since $\sqrt{\rho_{n}}$ converges strongly to $\sqrt{\rho}$ in $L^{\infty}(0,T;L^{2}(\Omega))$ and $\sqrt{\rho_{n}} u_{n}$ converges strongly to $\sqrt{\rho}u$ in $L^{2}(0,T;L^{2}(\Omega))$, the first term on the right-hand of \eqref{4.1} converges to the corresponding term in the sense of distribution. The converges of the second term on the right-hand side of \eqref{4.1} in the sense of distribution can be shown by the weak convergence $\nabla{\sqrt{\rho_{n}}}$ and the strong converges in $L^{2}(0,T;L^{2}(\Omega))$.

Step 3. Convergence of the entropy equation. In view of \eqref{3.16}-\eqref{3.18}, it is easy to pass to the limit $n\rightarrow\infty$ in all terms appearing in \eqref{1.12}, except the entropy production rate $\sigma$.

However, in accordance with \eqref{2.6} we still have that
\begin{equation*}
\{\sqrt{\frac{\rho_{n}}{\theta_{n}}}D(u_{n}) \}_{n=1}^{\infty}.
\end{equation*}
is bounded in $L^{2}((0,T)\times \Omega)$. Moreover, by virtue of \eqref{3.333}, \eqref{3.12} and \eqref{3.26} we deduce
\begin{equation*}
\sqrt{\frac{\rho_{n}}{\theta_{n}}}D(u_{n})\rightarrow \sqrt{\frac{\rho}{\theta}}D(u) .
\end{equation*}
Evidently, we may treat all the remaining terms
\begin{equation*}\label{0224}
\{\frac{\sqrt{\kappa(\rho_{n},\theta_{n})}}{\theta_{n}} \nabla \theta_{n} \}_{n=1}^{\infty}.
\end{equation*}
in the similar way using the fact that they are linear with respect to the weakly convergent sequences of gradients of $\rho_{n},u_{n},\theta_{n}$. Thus, preserving the sign of the entropy inequality \eqref{3.12} in the limit $n\rightarrow \infty$ follows by the lower semicontinuity of convex superposition of operators.

Step 4. Convergence of the total energy balance. It is straight to pass the limit $n\rightarrow \infty$ in the total energy balance.

\phantomsection

\end{document}